\documentclass[12pt,reqno]{amsart}

\usepackage{amsmath,amssymb,amsthm,amscd,amsfonts,amsbsy,bm}
\usepackage[english]{babel}
\usepackage[mathscr]{eucal}

\usepackage{graphicx}

\usepackage{xcolor}

\font\sc=rsfs10 at 12pt

\usepackage[active]{srcltx}
\usepackage[bookmarks=false,colorlinks]{hyperref}

\numberwithin{equation}{section}

\renewcommand{\a}{\alpha}
\renewcommand{\b}{\beta}
\newcommand{\g}{\gamma}

\renewcommand{\d}{\delta}

\newcommand{\z}{\zeta}

\renewcommand{\L}{\Lambda}
\newcommand{\m}{\mu}

\newcommand{\x}{\xi}

\newcommand{\s}{\sigma}

\newcommand{\f}{\phi}

\newcommand{\F}{\Phi}

\renewcommand{\o}{\omega}

\newcommand{\vs}{\varsigma}

\newcommand{\C}{{\mathbb C}}
\newcommand{\R}{{\mathbb R}}
\newcommand{\Z}{{\mathbb Z}}

\newcommand{\kb}{{\mathbf k}}

\newcommand{\nb}{{\mathbf n}}

\newcommand{\cF}{\mathfrak c}
\newcommand{\CF}{\mathfrak C}

\newcommand{\DF}{\mathfrak D}

\newcommand{\Ac}{{\mathcal A}}

\newcommand{\Ec}{{\mathcal E}}
\newcommand{\Fc}{{\mathcal F}}

\newcommand{\Pc}{{\mathcal P}}

\newcommand{\Dc}{\sc\mbox{D}\hspace{1.0pt}}

\newcommand{\supp}{\hbox{{\rm supp}}\,}

\newcommand{\dbar}{\overline{\partial}}

\newcommand{\Fp}{\pmb{\F}}
\newcommand{\pp}{\pmb{\partial}}

\newtheorem{theorem}{Theorem}[section]
\newtheorem{proposition}[theorem]{Proposition}
\newtheorem{lemma}[theorem]{Lemma}
\newtheorem{corollary}[theorem]{Corollary}

\theoremstyle{definition}
\newtheorem{definition}[theorem]{Definition}

\newtheorem{remark}[theorem]{Remark}
\newtheorem{example}[theorem]{Example}

\begin{document}

\title[Sesquilinear form symbols]{Toeplitz operators defined by sesquilinear forms: Fock space case}\thanks{The first named author is grateful to
CINVESTAV in Mexico City for support and hospitality.}\thanks{The second named author has been partially supported by CONACYT Project 180049,
Mexico, and Chalmers University of Technology, Gothenburg, Sweden. }

\author[Rozenblum]{Grigori Rozenblum}
\address{Department of Mathematics, Chalmers  University of Technology and The University of Gothenburg,   Gothenburg, Sweden}
\email{grigori@chalmers.se}

\author[Vasilevski]{Nikolai Vasilevski }
\address{Department of Mathematics, Cinvestav, Mexico City, Mexico}
\email{nvasilev@duke.math.cinvestav.mx}

\begin{abstract}
The classical theory of Toeplitz operators in spaces of analytic functions deals usually with symbols that are bounded measurable functions on
the domain in question. A further extension of the theory was made for symbols being unbounded functions, measures, and compactly supported
distributions, all of them subject to some restrictions.

In the context of a reproducing kernel Hilbert space we propose a certain framework for a `maximally possible' extension of the notion of
Toeplitz operators for a `maximally wide' class of `highly singular' symbols. Using the language of sesqui\-linear forms we describe a certain
common pattern for a variety of analytically defined forms which, besides  covering  all previously considered cases, permits us to
introduce a further substantial extension of a class of admissible symbols that generate bounded Toeplitz operators.

Although our approach is unified for all reproducing kernel Hilbert spaces, for concrete operator consideration  in this paper we restrict
ourselves to Toeplitz operators acting on the standard Fock (or Segal-Bargmann) space.
\end{abstract}
\maketitle
\section{Introduction}\label{se:intro}

The classical theory of Toeplitz operators in spaces of analytic functions (Hardy, Bergman,
Fock, etc spaces) deals usually with symbols that are bounded measurable functions on the
domain in question. As it was observed later, in the case of the Bergman and Fock spaces, for certain classes of unbounded symbols one still can
reasonably define  Toeplitz operators that prove to be bounded.

Further  natural extensions towards even more general symbols of Toeplitz operators were developed for the first time (to the best of our knowledge) in
\cite{Lue}, for symbols being measures, and in \cite {AlexRoz}, for symbols being compactly supported distributions. The main idea behind such
extensions was to enrich the class of Toeplitz
operators, and, in particular, to turn into Toeplitz many of operators that failed to be
Toeplitz in the above classical sense.

The aim of this paper is to propose a certain framework for a `maximally possible' extension of
the notion of Toeplitz operators for a `maximally wide' class of `highly singular' symbols, including the ones that involve derivatives of infinite order of measures. In
the general context of a reproducing kernel Hilbert space we propose a certain common pattern
which, besides  covering  all previously considered extensions, enables us to introduce a
further substantial extension of a class of `highly singular' but still admissible symbols that
generate bounded Toeplitz operators.

Our approach is based upon the extensive use of the language of sesqui\-linear forms. Let us
describe shortly its main ingredients (for more details, examples, and proofs see the main
text).
Let  $\mathcal{H}$ be a Hilbert space of functions defined in a  domain
$\Omega\subseteq\mathbb{R}^d$ or $\Omega\subseteq\mathbb{C}^d$, and let
$\mathcal{A}$ be its closed subspace having a reproducing kernel $\kb_z(w) \in \mathcal{A}$,
so that the orthogonal projection $P$ of $\mathcal{H}$ onto $\mathcal{A}$ has the form
\begin{equation*} 
 (Pu)(z) = \langle u(\cdot), \kb_z(\cdot) \rangle.
\end{equation*}
Then, given a bounded sesquilinear form $F(\cdot,\cdot)$ on $\mathcal{A}$, the Toeplitz
operator $T_F$ \emph{defined by the sesquilinear form} $F$ is the
operator which acts on $\mathcal{A}$ as
\begin{equation*}
 (T_F f)(z) = F(f(\cdot),\kb_z(\cdot)).
\end{equation*}
The next step in our study is to describe a certain common pattern for a variety of
analytically defined forms.
Throughout the paper we  consider various concrete classes of sesquilinear forms generated by
quite different analytical objects (functions, measures,
distributions, etc). For Toeplitz operators defined by such specific sesquilinear forms we say
that they have corresponding \emph{function,
measure, distribution, etc symbols}.
We mention as well that different analytical expressions may define \emph{the same} sesquilinear
form and thus define \emph{the same} Toeplitz operator.

Our common pattern in defining sesquilinear forms and the corresponding Toeplitz operators is as
follows. We start with a (complex) linear topological space $X$ (not necessarily complete). Let
$X'$ be its dual space (the set of all continuous linear
functionals on $X$), and denote by $\Phi(\phi) = (\Phi,\phi)$ the intrinsic pairing of $\Phi
\in X'$ and $\phi \in X$. Let finally $\mathcal{A}$
be a reproducing kernel Hilbert space with $\kb_z(\cdot)$ being its reproducing kernel
function.

By a continuous $X$-valued sesquilinear form $G$  on $\mathcal{A}$ we mean a continuous
mapping
\begin{equation*}
 G(\cdot,\cdot) \, : \ \mathcal{A} \oplus \mathcal{A}
 \longrightarrow \ X,
\end{equation*}
which is linear in the first argument and anti-linear in the second one.
Then, given a continuous $X$-valued sesquilinear form $G$ and an element $\Phi \in X'$, we
define the sesquilinear form $F_{G,\Phi}$ on
$\mathcal{A}$ by
\begin{equation*} 
 F_{G,\Phi}(f,g) = \Phi(G(f,g)) = (\Phi, G(f,g)).
\end{equation*}
Being continuous, this form is bounded, and thus defines a bounded Toeplitz operator
\begin{equation*} 
 (T_{G,\Phi}f)(z) := (T_{F_{G,\Phi}}f)(z) = (\Phi, G(f,\kb_z)).
\end{equation*}
Having such an extended approach to Toeplitz operators, we immediately gain, for example, the following
very important result (Theorem \ref {st:algeb_oper}): \emph{The set of Toeplitz operators is
a $^*$-algebra.} Recall in this connection that classical Toeplitz operators do not possess  such
a nice property.

Although our approach is unified for all reproducing kernel Hilbert spaces, for concrete
considerations  in this paper we restrict ourselves to Toeplitz operators acting in the
standard Fock (or Segal-Bargmann) space. We show that all  previously introduced extensions of the notion of Toeplitz operators to unbounded,
measure, and distributional symbols fit perfectly   our pattern. Moreover, we describe several further extensions to even more singular symbols.

We introduce and study symbols being Fock-Carleson measures for derivatives, for which we characterize boundedness and compactness of
corresponding Toeplitz operators (Proposition \ref{prop:integral}, Proposition \ref{prop:Toepl.Distr}, and Corollary \ref{Cor:vanishing}). In
particular, we show (Theorem \ref{th:weak-universal}) that given any bounded linear operator on the Fock space, any  finite truncation of the
infinite matrix representation  of this operator with respect to the standard basis in $\mathcal{F}^2(\mathbb{C})$ is a \emph{Toeplitz operator} with distributional
symbol supported at $\{0\}$, or (in other terms) with symbol being the sum of derivatives of certain Fock-Carleson measures.

In Section \ref{se:non-Toeplitz} we collect various examples of operators that fail to be
Toeplitz ones in the classical sense with bounded, unbounded (in a certain class), and even distributional symbols. All of
them, and many more, become  Toeplitz operators in our extended sense.

Our study is essentially based upon some  new sharp pointwise estimates for  derivatives of
functions  $f \in \mathcal{F}^2(\mathbb{C})$, with explicitly shown dependence of the
constants on the order of the differentiation (Proposition \ref{prop:est.deriv.1}, Proposition
\ref{prop:estimates.better}), and a more sharp local estimate, where the
$\Fc^2$-norm is replaced by an integral, involving $f$, over a certain disk
(Proposition~\ref{prop:est.deriv.2}).

We note finally that the results of the paper can be easily extended, with just minor technical changes in the proofs, to the case of operators
acting on the multidimensional Fock space $\mathcal{F}^2(\mathbb{C}^n)$, with~$n >1$.

\section{General operator theory approach to \\ Toeplitz operators}
\label{se:general}

To proceed with our plan we need to start from the very beginning and discuss anew the basic
notions and definitions.

Let $H$ be a Hilbert space with $H_0$ being its closed subspace. We denote by $P$ the
orthogonal projection of $H$ onto $H_0$.  In the most
general setting, given a linear bounded operator $A$ acting in $H$, the Toeplitz operator $T_A$
associated with  $A$ and acting in $H_0$
($\equiv$ \ the compression of $A$ onto $H_0$, or the angle of the operator $A$) is defined as
\begin{equation} \label{eq:compression}
 T_A\, : \ x \in H_0 \ \longmapsto \ P(Ax) \in H_0,  \ \ \ \mathrm{i.e.,} \ \ T_A = PA|_{H_0}.
\end{equation}

The interrelation between $A$ and $T_A$  is thus very simple: the latter operator is a
compression of the former, while the former is a dilation
of the latter.

In such a general setting, different operators $A'$ and $A''$ can obviously generate the same
Toeplitz operator. Indeed, let
\begin{equation*}
A' =
\begin{pmatrix}
 A'_{1,1} & A'_{1,2} \\
 A'_{2,1} & A'_{2,2}
\end{pmatrix}
 \ \ \ \ \mathrm{and} \ \ \ \
A'' =
\begin{pmatrix}
 A''_{1,1} & A''_{1,2} \\
 A''_{2,1} & A''_{2,2}
\end{pmatrix}
\end{equation*}
be the matrix representations of $A'$ and $A''$ in $H = H_0 \oplus H_0^{\perp}$. Then $T_{A'} =
T_{A''}$ if and only if $A'_{1,1} = A''_{1,1}$.

Thus, to have a substantial theory, one should consider some natural and important subclasses
of operators on $H$.

One of the most known and classical examples here is represented by Toeplitz operators with
bounded measurable symbols, acting in the Fock space.
In this example $H$ is $L_{2}(\C,d\nu)$, i.e.,  the Hilbert space of functions on $\mathbb{C}$, square-integrable
with respect to the Gaussian measure
\begin{equation*}
  d\nu(z)=\omega(z)dV(z), \ \ \ \ \ \ \mathrm{where} \ \ \ \ \omega(z) = \pi^{-1}\,
  e^{-z\cdot\overline{z}},
\end{equation*}
and $dV(z)=dxdy$ is the Lebesgue plane measure on
$\mathbb{C}=\R^{2}$. Then $H_0 = \mathcal{F}^2(\mathbb{C})$ is its Fock \cite{Ber,Fock32} (or
Segal--Bargmann
\cite{Bargmann61,Segal60}) subspace consisting of all functions analytical in $\mathbb{C}$, and
$P$ is the orthogonal projection of
$L_{2}(\mathbb{C},d\nu)$ onto  $\mathcal{F}^{2}(\mathbb{C})$. Given a function
$a=a(z) \in  L_{\infty}(\mathbb{C})$, the Toeplitz
operator $T_a$ with  symbol $a$ is the compression onto $\mathcal{F}^2(\mathbb{C})$ of the
multiplication operator $(M_a f)(z) = a(z)f(z)$ on
$L_2(\mathbb{C},d\nu)$:
\begin{equation*}
  T_a: f \in \mathcal{F}^2(\mathbb{C}) \longmapsto P(af) \in \mathcal{F}^2(\mathbb{C}).
\end{equation*}

It is well known (see, for example, \cite{BergCob86}) that in this case, i.e., if we restrict
the class of defining operators to the above
multiplication operators, the (function) symbol $a$ is uniquely defined by the Toeplitz
operator.

The aim of this paper is to extend the above setting to possibly most general symbols,
admitting unboundedness and various types of
singularities. Although the main results of the paper are given for the Toeplitz operators
acting in the Fock space $\Fc^2(\C)$, our approach is
applicable in a more general setting of a rather arbitrary Hilbert space with reproducing
kernel.

We recall that if  $\mathcal{H}$ is a Hilbert space of functions defined in a  domain
$\Omega\subseteq\mathbb{R}^d$ or $\Omega\subseteq\mathbb{C}^d$,
then its closed subspace  $\mathcal{A}$ is called a reproducing kernel subspace if any
evaluation functional $\mathcal{A}\ni f\mapsto f(z)$ is well
defined and bounded. Most typically such subspaces consist of $L_2$-solutions of elliptic equations or systems.  For any fixed $z \in \Omega$,
let $\kb_z(w) \in \mathcal{A}$ be the
element in $\mathcal{A}$ realizing by the Riesz theorem this
evaluation functional, i.e.,
\begin{equation*}
  f(z) = \langle f(w), \kb_z(w) \rangle, \ \ \ \ \mathrm{for \ all} \ \ \ f \in \mathcal{A},
\end{equation*}
so that the orthogonal projection $P$ of $\mathcal{H}$ onto $\mathcal{A}$ has the form
\begin{equation*}
  P \, : \ u(z) \in \mathcal{H} \ \longmapsto \ \langle u(w), \kb_z(w) \rangle \in
  \mathcal{A},
\end{equation*}
or
\begin{equation} \label{eq:projection}
 (Pu)(z) = \langle u(\cdot), \kb_z(\cdot) \rangle.
\end{equation}
Here and further on, by $\langle\cdot,\cdot\rangle$ we denote the scalar product in the Hilbert space under consideration.
The function $\kb_z(w)$ is called the reproducing kernel for $\Ac$

Recall that a bounded sesquilinear form $F(\cdot,\cdot)$ on $\mathcal{A}$ is a mapping
\begin{equation*}
 F(\cdot,\cdot) \, : \ \mathcal{A} \oplus \mathcal{A} \longrightarrow \ \mathbb{C},
\end{equation*}
which is linear in the first argument and anti-linear in the second one and, additionally,
satisfies the boundedness condition: there exists a
constant $C \geq 0$ such that
\begin{equation*}
 |F(f,g)| \leq C \|f\|\cdot \|g\|, \ \ \ \ \mathrm{for \ all} \ \ \ f,\, g \in \mathcal{A}.
\end{equation*}
By the Riesz theorem for bounded sesquilinear forms, for a given form $F(\cdot,\cdot)$, there
exists a unique bounded linear operator $T$ in $\mathcal{A}$
such that
\begin{equation} \label{eq:Riesz}
 F(f,g) = \langle Tf, g \rangle, \ \ \ \ \mathrm{for \ all} \ \ \ f,\, g \in \mathcal{A}.
\end{equation}

In this paper we  adopt  the following vocabulary.\\
Given a bounded sesquilinear form $F(\cdot,\cdot)$ on $\mathcal{A}$, the Toeplitz operator
$T_F$ \emph{defined by the sesquilinear form} $F$ is the
operator which acts on $\mathcal{A}$ as
\begin{equation}\label{eq:operator_viaForm}
 (T_F f)(z) = F(f(\cdot),\kb_z(\cdot)).
\end{equation}
The terminology ``Toeplitz'' is consistent with the general definition (\ref{eq:compression})
since
(\ref{eq:projection}) and (\ref{eq:Riesz}) imply that
\begin{equation*}
 (T_F f)(z) = \langle Tf, \kb_z \rangle = (Tf)(z) = (T_A f)(z)
\end{equation*}
for any dilation $A$ of the operator $T$ on $\mathcal{A}$ to some Hilbert space $\mathcal{H} =
\mathcal{A} \oplus \mathcal{A}^{\perp}$.

Note that although the dilation $A$ is not unique, the operator $T_F$ is uniquely defined by
the form $F$. We mention as well that quite
different analytic expressions may define \emph{the same} sesquilinear form and thus define
\emph{the same} Toeplitz operator. Examples for this
effect will be presented further on.

\begin{remark}
Throughout the paper we  consider various concrete classes of sesquilinear forms generated by
different analytic objects (functions, measures, distributions, etc). For Toeplitz operators defined by such specific sesquilinear forms we will
say that they have corresponding \emph{function, measure, distribution, etc symbols}.
\end{remark}

We illustrate now the above approach to Toeplitz operators by specifying some classes of
sesquilinear forms. In this paper the enveloping
Hilbert space $\mathcal{H}$ will  always be $L_2(\mathbb{C},d\nu)$ and its reproducing kernel
subspace $\mathcal{A}$ will  always be the Fock space
$\mathcal{F}^2(\mathbb{C})$ consisting of analytical functions. Recall that in this case
\begin{equation*}
 \kb_z(w) = e^{w\overline{z}},
\end{equation*}
so that the orthogonal  projection $P$ from $L_2(\mathbb{C},d\nu)$ onto
$\mathcal{F}^2(\mathbb{C})$ has the form
\begin{equation*}
 (Pu)(z) = \int_{\mathbb{C}} u(w)e^{z\overline{w}} d\nu(w)  = \frac{1}{\pi} \int_{\mathbb{C}}
 u(w)e^{(z-w)\overline{w}} dV(w).
\end{equation*}
It is impoprtant to keep in mind that the standard orthonormal monomial basis in $\mathcal{F}^2(\mathbb{C})$ has
the form
\begin{equation}\label{eq:basis}
 e_k(z) = \frac{1}{\sqrt{k!}}\, z^k, \ \ \ \ k \in \mathbb{Z}_+.
\end{equation}

\begin{example}\textsf{Classical Toeplitz operators on the Fock space.} \\
We start with an arbitrary \emph{bounded linear functional} $\Phi \in L^*_1(\mathbb{C},d\nu)$.
As well known, such a functional is uniquely
defined by a function $a \in L_{\infty}(\mathbb{C})$ and has the form
\begin{equation*}
 \Phi(u) = \Phi_a(u) = \int_{\mathbb{C}} a(z)u(z)d\nu(z),
\end{equation*}
with $\|\Phi_a\|=\|a\|_{L_\infty}.$ For any $f,\, g \in \mathcal{F}^2(\mathbb{C})$
the product $f\overline{g}$ belongs to
$L_1(\mathbb{C},d\nu)$, and finite linear combinations of such products are dense in
$L_1(\mathbb{C},d\nu)$. We define the sesquilinear forms
$F_a$ as
\begin{equation} \label{eq:F_a}
 F_a(f,g) = \Phi_a(f\overline{g}).
\end{equation}
This form is obviously bounded:
\begin{equation*}
 |F_a(f,g)| \leq \|\Phi_a\| \|f\overline{g}\|_{L_1} \leq \|a\|_{L_\infty} \|f\| \|g\|.
\end{equation*}
Then
\begin{eqnarray*}
 (T_{F_a} f)(z) &=& F_a(f,\kb_z) = \Phi_a(f\overline{\kb_z}) = \int_{\mathbb{C}}
 a(w)f(w)\overline{\kb_z(w)}d\nu(w)\\
 &=&  \int_{\mathbb{C}} a(w)f(w)\,e^{z\overline{w}}\, d\nu(w) = (T_af)(z),
\end{eqnarray*}
i.e, the Toeplitz operator $T_{F_a}$ generated by the sesquilinear form $F_a$ coincides with
the classical Toeplitz operator having the function symbol
$a$. Moreover, all Toeplitz operators with bounded measurable symbols can be obtained starting
with a  proper functional $\Phi$ in
$L^*_1(\mathbb{C},d\nu)$, which defines in turn the form (\ref{eq:F_a}).

It was already mentioned (in this setting) that the $L_{\infty}$-symbol $a(z)$ is uniquely defined
by a Toeplitz operator. If we admit  unbounded symbols
$a(z)$ then the corresponding sesquilinear forms $F_a$ in \eqref{eq:F_a} are not bounded in
$\mathcal{F}^2(\mathbb{C})$, in general. However,
for some classes of unbounded $a$, this sesquilinear form can still be well defined and
bounded.
Folland \cite[p. 140]{Folland89} extended the above  uniqueness result to the class of
unbounded symbols that satisfy the inequality
\begin{equation} \label{eq:Folland}
  |a(z)| \leq C e^{\delta |z|^2}, \ \ \ \ \mathrm{for\ some} \ \ \ \
  \delta <1.
\end{equation}
On the other hand, the further extension to wider classes of unbounded symbols may lead to the
failure of the uniqueness of symbols: \emph{there exist nontrivial symbols $a$ for which $T_a=0$}. See \cite[Theorem 3.4]{GrVas} and
\cite[Proposition 4.6]{BauLe} for concrete examples of such $a$.
\end{example}

\begin{example}\textsf{Toeplitz operators defined by Fock-Carleson measures.} \\
Recall (see, for example, \cite[Section 3.4]{ZhuFock}) that a finite positive Borel measure
$\mu$ on $\mathbb{C}$ is called a Fock-Carleson
measure (FC measure) for the space $\mathcal{F}^2(\mathbb{C})$ if there exists a constant $C >
0$ such that
\begin{equation} \label{eq:Carleson}
 \int_{\mathbb{C}} |f(z)|^2 e^{-|z|^2}d\mu(z) \leq C \int_{\mathbb{C}} |f(z)|^2 d\nu(z), \ \ \
 \ \mathrm{for \ all} \ \  f \in
 \mathcal{F}^2(\mathbb{C}).
\end{equation}
We define the  sesquilinear form $F_{\mu}$ as
\begin{equation*}
 F_{\mu}(f,g) = \int_{\mathbb{C}} f(z)\overline{g(z)}e^{-|z|^2} d\mu(z),
\end{equation*}
which is obviously bounded by the Cauchy inequality and (\ref{eq:Carleson}). Then
\begin{eqnarray*}
 (T_{F_{\mu}} f)(z) &=& \int_{\mathbb{C}} f(w)\overline{\kb_z(w)}e^{-|w|^2}d\mu(w) \\
 &=& \int_{\mathbb{C}} f(w)e^{(z-w)\overline{w}}d\mu(w) = (T_{\mu}f)(z),
\end{eqnarray*}
i.e, the Toeplitz operator $T_{F_\mu}$ generated by the sesquilinear form $F_{\mu}$ is nothing
but the (bounded) Toeplitz operator defined by the
Fock-Carleson measure $\mu$.
\end{example}

\begin{example}\textsf{Toeplitz operators with compactly supported distributional symbols.} \\
Any distribution $\Phi$ in $\Ec'(\mathbb{C})$ has  finite order, and thus can be extended to a
continuous functional in the space of functions
with finite smoothness,
\begin{equation*}
    |\Phi(h)|\le C(\Phi)\|h\|_{C^N(\mathbb{K})}, \ \ \ \  h \in C^N(\mathbb{K}),
\end{equation*}
for some compact set $\mathbb{K}\subset \mathbb{C}$ containing the support of $\Phi$.
It follows from the Cauchy formula  that the $C^N(\mathbb{K})$-norm of the product
$h(z)=e^{-z\overline{z}}f(z)\overline{g(z)}$, $f,g\in
\mathcal{F}^2(\mathbb{C})$, is majorated by the product of the
$\mathcal{F}^2(\mathbb{C})$-norms of $f$ and $g$. Therefore, the sesquilinear
form
\begin{equation*}
 F_{\Phi}(f,g) = \Phi\left(\omega(\cdot)f(\cdot)\overline{g(\cdot)}\right)
\end{equation*}
is bounded in $\mathcal{F}^2(\mathbb{C})$ and thus defines a bounded operator.

The operator
\begin{equation}\label{eq:dist_Phi}
 (T_{F_{\Phi}}f)(z) = \Phi\left(\omega(\cdot)f(\cdot)\overline{\kb_z(\cdot)}\right) =
 (T_{\Phi}f)(z)
\end{equation}
generated by the sesquilinear form $F_{\Phi}$ coincides with the standardly defined
\cite{AlexRoz,RozToepl} Toeplitz operator $T_{\Phi}$ with
distributional symbol $\Phi$.
By the structure theorem for  distributions with compact support (see, e.g., \cite[Section
4.4]{Gelf}), $\Phi$ admits a representation as a
\emph{finite} sum
\begin{equation}\label{Struct}
    \Phi=\sum_{q}D^q u_q,
\end{equation}
where  $q=(q_1,q_2)$, $D=(D_1,D_2)$ is the distributional gradient, and $u_q$ are continuous functions that can be chosen as
having support in an arbitrarily close
neighborhood of the support of $\Phi$. Rearranging the derivatives, we can rewrite \eqref{Struct}
as the finite sum
\begin{equation}\label{Struct2}
    \Phi=\sum_{q}\partial_z^{q_1}\partial_{\overline{z}}^{q_2}\,v_q,
\end{equation}
again with certain continuous functions $v_q$, having support in an arbitrarily close
neighborhood of $\supp \Phi$. Then  \eqref{eq:dist_Phi}
transforms to
\begin{gather}\label{Action}
    (T_{\Phi}f)(z)=\sum_q (-1)^{q_1+q_2}\frac{1}{\pi}\int_{\mathbb{C}} v_q(w)\cdot
    \partial_w^{q_1}\partial_{\overline{w}}^{q_2}\left[e^{(z-w)\overline{w}}f(w)\right]dV(w).
\end{gather}
\end{example}
We present now several examples of the action of such operators $T_{\Phi}$ with
$\Phi\in\Ec'(\C)$. By $\d$ we denote here the usual
$\d$-distribution in $\R^2=\C$, centered at $0$.

Let $\Phi_{p,q} = \partial_w^p \partial_{\overline{w}}^q \,\delta$. Then, by
(\ref{eq:dist_Phi}), we have
\begin{equation}\label{Action2}
 T_{\Phi_{p,q}}\, e_k = \frac{(-1)^{p+q}}{\pi}\, \delta \left(\partial_w^p
 \partial_{\overline{w}}^q \left(e_k(w)\,e^{(z-w)\overline{w}}
 \right)\right),
\end{equation}
where $e_k$ is the basis \eqref{eq:basis}. Simple though lengthy calculations show that
\begin{equation}\label{Action3}
 T_{\Phi_{p,q}}\, e_k =
\left\{
\begin{array}{ll}
 \frac{(-1)^{q-k}}{\pi} \frac{p!\,q!}{\sqrt{k!(q-p+k)!}\,k!(p-k)!}\, e_{q-p+k}, & \mathrm{if} \
 \max(0,p-q)\leq k \leq p, \\
 0, & \mathrm{otherwise}.
\end{array} \right.
\end{equation}

For  $q = p$ we write $\Phi_p = \Phi_{p,p}$, and in this case
\begin{equation}\label{Action4}
 T_{\Phi_p}\, e_k =
\left\{
\begin{array}{ll}
 \frac{(-1)^{p-k}}{\pi} \left[\frac{p!}{k!}\right]^2 \frac{1}{(p-k)!}\, e_k, & \mathrm{if} \ 0
 \leq k \leq p, \\
 0, & \mathrm{otherwise}.
\end{array} \right.
\end{equation}
Thus, the Toeplitz operator $T_{\Phi_p}$ is a linear combination of the rank one projections
$P_k = \langle \cdot, e_k \rangle e_k$, for $k
=0,1,...,p$. Vice versa, the orthogonal projection $P_n$ is a linear combination of the
Toeplitz operators $T_{\Phi_p}$, for $p =0,1,...,n$,
i.e., $P_n$ is a Toeplitz operator with a certain distributional symbol having the one-point
support $\{0\}$.

Let now $p \neq q$. If $p < q$, then the Toeplitz operator $T_{\Phi_{p,q}}$ is a linear
combination of the rank one operators $P_{n,m} = \langle
\cdot, e_n \rangle e_m$ with $n = k$, $m = q-p+k$, and $k=0,1,...,p$. Vice versa: the rank one
operator $P_{p,q}$ is a linear combination of the
Toeplitz operators $T_{\Phi_{p-n,q-n}}$, with $n =0,1,...,p$.

If $p > q$, then the Toeplitz operator $T_{\Phi_{p,q}}$ is a linear combination of the rank one
operators $P_{p-q+n,n}$ with $n = 0,1,...,q$.
Vice versa: the rank one operator $P_{p,q}$ is a linear combination of the Toeplitz operators
$T_{\Phi_{p-n,q-n}}$, with $n =0,1,...,q$.

We note that for $p \neq q$, again, each rank one operator $P_{p,q}$ is a Toeplitz operator
with a certain distributional symbol having the
one-point support $\{0\}$.

\vskip0.2cm
The above representation of the  operators  $P_{p,q}$, $p,\,q \in \mathbb{Z}_+$, does not look
quite satisfactory: a rank one operator
corresponds to several terms in the distributional symbol. It is more convenient and aesthetic
to use another, more straightforward, form:
 the distributional symbol
\begin{equation} \label{eq:Phi_p,q}
 \Psi_{p,q} = \frac{(-1)^{p+q}}{\sqrt{p!\,q!}}\,\omega^{-1}(w)\, \partial_w^p
 \partial_{\overline{w}}^q \,\delta.
\end{equation}
Then, by (\ref{eq:dist_Phi}), we have
\begin{equation*}
 T_{\Psi_{p,q}}\, e_k =
\left\{
\begin{array}{ll}
 e_q, & \mathrm{if} \ k = p, \\
 0, & \mathrm{otherwise}.
\end{array} \right.
\end{equation*}
Thus $P_{p,q} = T_{\Psi_{p,q}}$, for all  $p,\,q \in \mathbb{Z}_+$.

Either of the above representations of $P_{p,q}$ implies now the following lemma.

\begin{lemma} \label{le:P_p,q}
 For each $p,\,q \in \mathbb{Z}_+$ the rank one operator $P_{p,q}$ is a Toeplitz operator with
 a certain distributional symbol having the
 one-point support $\{0\}$.
\end{lemma}

We mention also that finite linear combinations of the rank one Toeplitz operators $P_{p,q}$,
where $p,q \in \mathbb{Z}_+$, form a norm dense
subset both in the set of all finite rank and all compact operators on
$\mathcal{F}^2(\mathbb{C})$. We improve the understanding of this effect
further on, in Section 6.

\section{Non-Toeplitz operators in $\mathcal{F}^2(\mathbb{C})$} \label{se:non-Toeplitz}

In this section we present examples of  operators in the Fock space
$\mathcal{F}^2(\mathbb{C})$, that are bounded but are not Toeplitz operators with
classical functional or even distributional symbols. Even certain classes of unbounded symbols cannot generate these operators.  We begin  with
their definitions.

\subsection{Classes of unbounded symbols}

{\sf Class $\DF_{1,-}$.} This class consists of all (possibly unbounded) measurable functions
$a$ satisfying the condition
\begin{equation*}
 a(\cdot)\kb_z(\cdot) \in L_2(\mathbb{C}, d\nu), \ \ \ \ \mathrm{for \ all} \ \   z \in
 \mathbb{C}.
\end{equation*}
Such functions were already used in \cite{BergCob94}, while the notation was introduced in
\cite{BorRoz}.

\bigskip
{\sf Class $\mathcal{D}_\cF$.} This class consists of all (possibly unbounded) measurable
functions $a$ satisfying the condition
\begin{equation*}
 \exists\, d > 0 \ \ \ \ \mathrm{such \ that} \ \ \ |a(z)| \leq de^{\cF|z|^2} \ \ \
 \mathrm{a.e.}  \ \ z \in \mathbb{C}.
\end{equation*}
Such functions were already used in \cite{Folland89} (see also (\ref{eq:Folland})), while the
notation was introduced in \cite[Section
2]{BauLe}.

\bigskip
{\sf Class $L_1^{\infty}(\R_+, e^{-r^2})$.} This class was introduced in \cite[Section
3]{GrVas} and consists of all (possibly unbounded)
measurable radial functions $a = a(r)$, $r = |z|$, satisfying the condition
\begin{equation*}
  \int_{\R_+} |a(r)|\, e^{-r^2}\, r^n\, dr < \infty, \ \ \ \ \mathrm{for \ all} \ \   n \in
  \mathbb{Z}_+.
\end{equation*}

\subsection{Radial operators}
\begin{example}
Let $(J\varphi)(z) = \varphi(-z)$ be the reflection operator in $\mathcal{F}^2(\mathbb{C})$. It
is obviously bounded and acts in the standard
monomial basis \eqref{eq:basis} of $\mathcal{F}^2(\mathbb{C})$
as follows:
\begin{equation}\label{reflection}
 (Je_k)(z) = (-1)^k e_k(z).
\end{equation}
That is, $J$ is a diagonal  operator with respect to the above basis (and thus \emph{radial}, see \cite{Zorboska03}), and its  eigenvalue
sequence,
ordered in accordance with \eqref{reflection}, has the form $\gamma_J = \{(-1)^k\}_{k \in
\mathbb{Z}_+}$.

By \cite[Theorem 17]{BergCob94} the operator $J$ cannot be a Toeplitz operator with
$L_{\infty}$-symbol and even with any (unbounded) symbol $a
\in \DF_{1,-}$. Moreover, $J$ cannot be norm approximated by Toeplitz operators with symbols in
$\DF_{1,-}$: by \cite[Theorem 17]{BergCob94} for
any symbol $a\in\DF_{1,-}$, the norm  $\|J - T_a\| $ is at least $1$.

On the other hand, if we allow unbounded radial symbols in\\  $L_1^{\infty}(\R_+, e^{-r^2})$,
i.e., with just slightly weaker growth
restrictions, then by \cite[Theorem 3.7]{GrVas} there exists a symbol $a_J$ in this class such
that $J = T_{a_J}$. Note that the proof of Theorem
3.7 in \cite{GrVas} gives an algorithm of constructing such a symbol.
\end{example}

\begin{example}
Consider in $\mathcal{F}^2(\mathbb{C})$ the orthogonal projection $P_0 f = \langle f, e_0
\rangle e_0$ onto the one-dimensional subspace
generated by $e_0$. It is a  diagonal, and thus \emph{radial}, operator having the eigenvalue
sequence $\gamma_{P_0} = (1,0,0,...)$.

Let us show now that $P_0$ cannot be a Toeplitz operator with a bounded measurable symbol and
even with an unbounded one in the class
$\mathcal{D}_\cF$, for any $\cF \in (0,1)$. Indeed, let us suppose that $P_0 =T_{a_0}$ with $a_0 \in
\mathcal{D}_\cF$, and $\cF \in (0,1)$. Since $P_0$ is
a projection, we have
\begin{equation*}
 0 = P_0 (I-P_0) = T_{a_0}(I - T_{a_0}) = T_{a_0}T_{1-a_0},
\end{equation*}
where both symbols $a_0$ and $1-a_0$ are radial and belong to $\mathcal{D}_\cF$. However, by
\cite[Theorem 4.4]{BauLe}, if the product of two Toeplitz operators with
$\mathcal{D}_\cF$-symbols is zero then one of them must be zero. This means that either $a_0 = 0$ or $a_0 = 1$, which is impossible for any
non-trivial projection.

The same  argument implies the following statement.
\begin{proposition}
 Neither  non-trivial radial projection can be represented as a Toeplitz operator with symbol
 from $\mathcal{D}_\cF$, for some $\cF \in (0,1)$. However,
 any such projection is a Toeplitz operator with  $L_1^{\infty}(\R_+, e^{-r^2})$-symbol.
\end{proposition}

On the other habd, the operator $P_0$ can be represented as a Toeplitz operator with just slightly  more
singular symbol, moreover such a representation is
not unique.
The first such representation of $P_0$ is as follows. Consider the distributional symbol
$\Phi_0 = \frac{1}{\pi}\delta$. Then, by (\ref{eq:dist_Phi}), for the Toeplitz operator $T_{\Phi_0}$ we have
\begin{equation*}
 T_{\Phi_0}e_k =  \left\{
\begin{array}{ll}
 1, & k = 0, \\
 0, & \mathrm{otherwise}.
\end{array} \right.
\end{equation*}
The Toeplitz operator $T_{\Phi_0}$, thus defined, is nothing but the above rank one projection
$P_0$. The second representation of $P_0$ comes
via symbols in $L_1^{\infty}(\R_+, e^{-r^2})$. By \cite[Theorem 3.7]{GrVas} there exists a
symbol $a_{\Phi_0} \in L_1^{\infty}(\R_+, e^{-r^2})$
such that $P_0 = T_{a_{\Phi_0}}$.

We mention also that, although the operator $P_0$ cannot be represented as a Toeplitz operator
with a bounded measurable radial symbol, it can be
norm approximated  by Toeplitz operators with such symbols:
\begin{equation*}
 P_0 = \lim_{n \to \infty} T_{a_n},
\end{equation*}
where $a_n(r) = (1+n)e^{-nr^2}$.

Indeed, the eigenvalue sequence of the operator $T_{a_n}$ \cite[Theorem 3.1]{GrVas} has the
form
\begin{eqnarray*}
 \gamma_{a_n}(k) &=& \frac{1}{k!}\int_{\mathbb{R}_+} (1+n)e^{-nr} e^{-r}r^kdr\\
 &=& \frac{1}{k!}\int_{\mathbb{R}_+} \frac{1}{(1+n)^k}\,e^{-s} s^kdrs = \frac{1}{(1+n)^k}.
\end{eqnarray*}
Thus
\begin{equation*}
 \|T_{a_n} - P_0\| = \|\gamma_{a_n} - \gamma_{P_0} \|_{\ell_{\infty}} = \gamma_{a_n}(1) =
 \frac{1}{1+n},
\end{equation*}
which implies the desired: $P_0 = \lim_{n \to \infty} T_{a_n}$.
\end{example}

\subsection{Finite rank operators}
Quite recently, a number of results concerning the characterization of finite rank operators in
the Fock space were obtained, see
\cite{AlexRoz,BorRoz,RozSW,RozShir}. These results produce more examples of operators which are
not Toeplitz ones, with functional or
distributional symbols.

By the main theorem in \cite{BorRoz}, any finite rank Toeplitz operator with symbol-function in
the class $\DF_{1,-}$ must be zero.  This leads
us to the following statement.
\begin{example}\label{ex.FinRank Func}
Let $T$ be a nonzero finite rank operator in the space $\mathcal{F}^2(\mathbb{C})$. Then $T$ is
\emph{not} a Toeplitz operator with
$\DF_{1,-}$-symbol.
\end{example}
Finite rank operators with distributional symbols have been considered in
\cite{AlexRoz,RozShir}. By the main theorem in \cite{AlexRoz}, for any
finite rank Toeplitz operator in $\mathcal{F}^2(\mathbb{C})$ with symbol $\Phi \in \Ec'(\C)$, i.e.,
$\Phi$ being a distribution with compact support,
this symbol must be a finite linear combination of $\d$-distributions and their derivatives.
Relations \eqref{eq:operator_viaForm},
\eqref{Action} show that for such a distribution, the range of the Toeplitz operator may contain
only linear combinations of finitely many
functions of the form $f_{w_j,k_j}(z)=z^{k_j}e^{z\overline{w_j}}$, with some integers $k_j$ and numbers $w_j\in\C$.

\begin{example}\label{ex.distr.comp} Let $T$ be a finite rank operator in
$\mathcal{F}^2(\mathbb{C})$ such that at least one function in the
range of $T$ is not a finite linear combination of functions $f_{w_j,k_j}(z)$. Then $T$ is
\emph{not} a Toeplitz operator with distributional
symbol in $\Ec'(\C)$. In particular, neither finite rank operator, with at least one function
having super-exponential growth in its range, can be
Toeplitz with a distributional symbol in $\Ec'(\C)$.  On the other hand, a finite rank operator
with at least one non-polynomial function in the
range, having a sub-exponential growth, cannot be a Toeplitz operator in $\Fc^2$ with symbol in
$\Ec'(\C)$ either. Examples for both cases can be
easily constructed using the Weierstrass product formula for entire functions with prescribed
zeros.

 One more, very simple and transparent illustration of the above effect  is as follows.
 Consider  two functions in  $\mathcal{F}^2(\mathbb{C})$:
\begin{equation*}
 e_n(z) = \frac{z^n}{\sqrt{n!}} \ \ \ \ \ \mathrm{and} \ \ \ \ \  \phi(z) = \sum_{k=0}^{\infty}
 \frac{z^k}{(k!)^{k +\frac{1}{2}}} =
 \sum_{k=0}^{\infty} \frac{1}{(k!)^k}\,e_k(z).
\end{equation*}
The function $\phi$ cannot be represented as a finite linear combination of functions
$f_{w_j,k_j}$ since its Fourier coefficients (in the basis
$\{e_k\}$) tend to zero much faster then those of any finite linear combination of functions
$f_{w_j,k_j}$.
Thus the rank one operator $P_{e_n,\phi}f = \langle f,e_n \rangle \phi$ is \emph{not} a
Toeplitz operator with distributional symbol in $\Ec'(\C)$.

It is important to note that the operator $P_{e_n,\phi}$ can be norm approximated by Toeplitz
operators with distributional symbol in $\Ec'(\C)$.
Indeed, let
\begin{equation*}
 \phi_m = \sum_{k=0}^m \frac{1}{(k!)^k}\,e_k \ \ \ \ \mathrm{and} \ \ \ \
  S_m = P_{e_n,\phi_m} = \sum_{k=0}^m \frac{1}{(k!)^k} P_{n,k},
\end{equation*}
then
\begin{equation*}
 \|P_{e_n,\phi} - S_m\| = \|\phi -\phi_m\|_{\mathcal{F}^2(\mathbb{C})} \to 0
 \ \ \ \ \mathrm{as} \ \ \ \ m \to \infty,
\end{equation*}
and the statement follows from Lemma \ref{le:P_p,q}.

\end{example}

In \cite{RozShir}, a similar  finite rank theorem was proved for a wider class $\L_{q,m}'$ of
symbols being distributions without compact support
condition but with certain growth restriction. Not repeating the definition of these classes,
we just refer to the main theorem in  \cite{RozShir}.
\begin{example}\label{ex.distr.noncomp} Let $T$ be a finite rank operator in
$\mathcal{F}^2(\mathbb{C})$ such that at least one function in the
range of $T$ is not a finite linear combination of functions $f_{w_j,k_j}(z)$. Then $T$ is
\emph{not} a Toeplitz operator with distributional
symbol in $\L_{q,m}'$.

The operator $P_{e_n,\phi}$ in the previous example may also serve as an illustration for
this case.
\end{example}

\section{General sesquilinear forms}
\label{se:construction}

Our aim in this paper is to consider Toeplitz operators with possibly most general defining sesquilinear
forms, so that they include all the above examples and
more.
Having this aim in view, we introduce the following construction. 

Let $X$ be a (complex) linear topological space (not necessarily complete). We denote by $X'$
its dual space (the set of all continuous linear
functionals on $X$), and denote by $\Phi(\phi) = (\Phi,\phi)$ the intrinsic pairing of $\Phi
\in X'$ and $\phi \in X$. Let finally $\mathcal{A}$
be a reproducing kernel Hilbert space with $\kb_z(\cdot)$ being its reproducing kernel
function.

By a continuous $X$-valued sesquilinear form $G$  on $\mathcal{A}$ we mean a continuous
mapping
\begin{equation*}
 G(\cdot,\cdot) \, : \ \mathcal{A} \oplus \mathcal{A}
 \longrightarrow \ X,
\end{equation*}
which is linear in the first argument and anti-linear in the second one.

Then, given a continuous $X$-valued sesquilinear form $G$ and an element $\Phi \in X'$, we
define the sesquilinear form $F_{G,\Phi}$ on
$\mathcal{A}$ by

\begin{equation} \label{eq:formF(G,Psi)}
 F_{G,\Phi}(f,g) = \Phi(G(f,g)) = (\Phi, G(f,g)).
\end{equation}
Being continuous, this form is bounded, and thus defines a bounded Toeplitz operator
\begin{equation} \label{eq:T_F(G,Psi)}
 (T_{G,\Phi}f)(z) := (T_{F_{G,\Phi}}f)(z) = (\Phi, G(f,\kb_z)).
\end{equation}

As it was already mentioned (and will be shown explicitly in the examples of this section), the
sesquilinear form that defines a Toeplitz operator may have several quite different analytic
expressions (involving different spaces $X$ and functionals $\Phi$), producing, nevertheless, the
same Toeplitz operator.

The algebraic operations with the above defined Toeplitz operators can also be described using
the language of sesquilinear forms. We start with
two Toeplitz operators $T_1 = T_{G_1,\Phi_1}$ and $T_2 = T_{G_2,\Phi_2}$,
with certain (complex) linear topological spaces $X_k$, $G_k$ being continuous $X_k$-valued
sesquilinear forms, and $\Phi_k \in  X'_k$, $k=1,2$.

Then (in one of suitable representations of the sesquilinear form)
\begin{itemize}
 \item the sum $T_1 + T_2$ is the Toeplitz operator $T = T_{G,\Phi}$ defined by the following
     data
\begin{equation*}
 X = X_1 \times X_2, \ \ \ G = (G_1, G_2), \ \ \ \Phi = (\Phi_1, \Phi_2) \in X';
\end{equation*}

 \item the product $T_1 T_2$ is the Toeplitz operator $T = T_{G,\Phi}$ defined by the
     following data
\begin{equation*}
 X = X_1, \ \ \ G(f,g) = G_1\left((\Phi_2,G_2(f,\kb_z)), g\right), \ \ \ \Phi = \Phi_1.
\end{equation*}
\end{itemize}

We note as well that if an operator $T$ is defined by a bounded sesquilinear form $F(f,g)$ then
its adjoint $T^*$ is defined by the transposed form
\begin{equation*}
 F^t(f,g) = \overline{F(g,f)},
\end{equation*}
that is, the adjoint operator to (\ref{eq:T_F(G,Psi)}) is defined by
\begin{equation*}
 (T_{F_{G,\Phi}}^*f)(z) = \overline{(\Phi, G(\kb_z,f))}.
\end{equation*}
In particular, the operator (\ref{eq:T_F(G,Psi)}) is self-adjoint if and only if its defining
form (\ref{eq:formF(G,Psi)}) is Hermitian
symmetric.

Summing up the above we arrive at the following statement.

\begin{theorem} \label{st:algeb_oper}
The set of Toeplitz operators of the form {\rm (\ref{eq:T_F(G,Psi)})} is a $^*$-algebra.
\end{theorem}

Note that although all our definitions were given in the context of any reproducing kernel Hilbert space, we apply here this construction only for the operators acting
in the Fock space $\mathcal{F}^2(\mathbb{C})$.
We return thus to the Fock space $\mathcal{F}^2(\mathbb{C})$ and to the examples given in the introductory part of the paper, and show that all of
them (and many more) fit into the above pattern.

\begin{example} \textsf{Classical Toeplitz operators.} \\
In this case $X = L_1(\mathbb{C}, d\nu)$, $X' = L_{\infty}(\mathbb{C})$, $G(f,g) =
f\overline{g}$, $\Phi= \Phi_a = a \in L_{\infty}(\mathbb{C})$,
so that
\begin{equation*}
 F_{G,\Phi_a}(f,g) = \int_{\mathbb{C}} a(z)f(z)\overline{g(z)}d\nu.
\end{equation*}\end{example}
\begin{example} \textsf{The reflection operator $J$.}
In this case $X = L_1(\mathbb{C}, d\nu)$, $X' = L_{\infty}(\mathbb{C})$, $G(f,g)(z) =
f(-z)\overline{g(z)}$, $\Phi= \Phi_1 = 1 \in
L_{\infty}(\mathbb{C})$, so that
\begin{equation*}
 F_{G,\Phi_1}(f,g) = \int_{\mathbb{C}} f(-z)\overline{g(z)}d\nu.
\end{equation*}
\end{example}

A natural extension of this example is as follows.

\begin{example} \textsf{Composition operators.}
They are the operators of the form $\CF_{\varphi}f = f \circ \varphi$, where $\varphi(z)$ is an
analytic map of $\varphi:\mathbb{C}\to \C$. By the results of \cite{CMS}, the operator $\CF_{\varphi}$ is
bounded if and only if $\varphi$ is a linear mapping, $\varphi(z) = az + b$ with $a,\, b \in \mathbb{C}$
and either  $|a| = 1$ and $b = 0$ or $|a| < 1$. Moreover, it is only in the latter case that  the operator
$\CF_{\varphi}$ is compact.

The corresponding sesquilinear form is the following: $X = L_1(\mathbb{C}, d\nu)$, $X' =
L_{\infty}(\mathbb{C})$, $G(f,g)(z) = f(az+b)\overline{g(z)}$, $\Phi= \Phi_1 = 1 \in
L_{\infty}(\mathbb{C})$, so that
\begin{equation*}
 F_{G,\Phi_1}(f,g) = \int_{\mathbb{C}} f(az+b)\overline{g(z)}d\nu.
\end{equation*}
Note that if $\varphi = e^{i\theta}z$ then the composition operator $\CF_{\varphi}:f(z) \mapsto f(e^{i\theta}z)$
is radial with the eigenvalue sequence $\gamma_{\CF_{\varphi}} = \{e^{in\theta}\}_{n \in
\mathbb{Z}_+}$.
\end{example}

\begin{example} \textsf{Toeplitz operators defined by FC measures.} \\
 In this case $X = L_1(\mathbb{C}, d\mu)$, $X' = L_{\infty}(\mathbb{C}, d\mu)$, $G(f,g)(z) =
 f(z)\overline{g(z)}e^{-|z|^2}$, $\Phi= \Phi_1 = 1
 \in L_{\infty}(\mathbb{C}, d\mu)$, so that
\begin{equation*}
 F_{G,\Phi_1}(f,g) = \int_{\mathbb{C}} f(z)\overline{g(z)}e^{-|z|^2}d\mu.
\end{equation*}
A natural generalization of this situation is to admit a complex valued Borel measure $\mu$
such that its variation $|\mu|$ is an FC measure. In
such a case $X = L_1(\mathbb{C}, d\mu) := L_1(\mathbb{C}, d|\mu|)$, $X' = L_{\infty}(\mathbb{C},
d\mu) := L_{\infty}(\mathbb{C}, d|\mu|)$ with the same
formulas for $G(f,g)$ and $\Phi$ as before.
The corresponding Toeplitz operators have the form
\begin{equation*}
 (T_{F_{G,\Phi_1}}f)(z) = F_{G,\Phi_1}(f,\kb_z) =\int_{\mathbb{C}}
 f(w)e^{(z-w)\overline{w}}d\mu(w).
\end{equation*}
In particular, this includes the case of a positive FC measure $\mu$,
$X = L_1(\mathbb{C}, d\mu)$, $X' = L_{\infty}(\mathbb{C}, d\mu)$, $G(f,g)(z) =
f(z)\overline{g(z)}e^{-|z|^2}$, $\Phi= \Phi_a = a \in
L_{\infty}(\mathbb{C}, d\mu)$, so that
\begin{equation*}
 F_{G,\Phi_a}(f,g) = \int_{\mathbb{C}} a(z) f(z)\overline{g(z)}e^{-|z|^2}d\mu.
\end{equation*}
and
\begin{equation*}
 (T_{F_{G,\Phi_a}}f)(z) =  \int_{\mathbb{C}} a(w) f(w)e^{(z-w)\overline{w}}d\mu(w).
\end{equation*}
\end{example}

\begin{example} \textsf{Toeplitz operators with compactly supported distributional symbols.}
\\
 In this case $X = \Ec(\mathbb{C})$, $X' = \Ec'(\mathbb{C})$, $G(f,g) = \omega f\overline{g}$,
 $\Phi \in \Ec'(\mathbb{C})$, so that
\begin{equation*}
 F_{G,\Phi}(f,g) =  \Phi\left(\omega(\cdot)f(\cdot)\overline{g(\cdot)}\right) = (\Phi, \omega f
 \overline{g}).
\end{equation*}
\end{example}

\begin{example} \textsf{Finite rank operators.} \\
We start with the rank one operators $P_{p,q}f = \langle f, e_p \rangle e_q$, $p,\,q \in
\mathbb{Z}_+$. They can be defined at least by the
following two sesquilinear forms. \\
For the first one, $X = \Ec(\mathbb{C})$, $X' = \Ec'(\mathbb{C})$, $G(f,g) = \omega
f\overline{g}$, and  $\Phi  = \Phi_{p,q} \in
\Ec'(\mathbb{C})$ is given by (\ref{eq:Phi_p,q}), so that
\begin{equation*}
 F_{G,\Phi_{p,q}}(f,g) =  (\Phi_{p,q}, \omega f \overline{g}).
\end{equation*}
In the second case, $X = \mathbb{C}$, $X' = \mathbb{C}$, $G(f,g) = \langle f, e_p \rangle
\langle e_q, g \rangle$, $\Phi= \Phi_1 = 1 \in
\mathbb{C}$, so that
\begin{equation*}
 F_{G,\Phi_1}(f,g) = \langle f, e_p \rangle \langle e_q, g \rangle.
\end{equation*}
Any finite rank operator in the Fock space has the form
\begin{equation*}
 Rf = \sum_{j=1}^K \langle f, \phi_j \rangle \psi_j,
\end{equation*}
where $\phi_j, \, \psi_j$, $j=1,...,K$, are functions in $\mathcal{F}^2(\mathbb{C})$.
To define the operator $R$ via a sesquilinear form, we set $X = \mathbb{C}^K$, $X' =
(\mathbb{C}^K)'$, identified with $\C^K$ by means of the
standard pairing. Further on, we set
 $G(f,g) = \left(\langle f, \phi_1 \rangle \langle \psi_1, g \rangle, ... , \langle f, \phi_K
 \rangle \langle \psi_K, g \rangle \right)$, $\Phi=$
 $  \Phi_{\mathbf{1}} = (1,...,1) \in \mathbb{C}^K$, so that
\begin{equation*}
 F_{G,\Phi_\mathbf{1}}(f,g) = \sum_{j=1}^K \langle f, \phi_j \rangle \langle \psi_j, g
 \rangle.
\end{equation*}
Another representation of this form can be obtained by setting \\
$X = (\mathcal{F}^2(\mathbb{C})\otimes \overline{\mathcal{F}^2(\mathbb{C})})^K$, $X'
=(\mathcal{F}^2(\mathbb{C})\otimes
\overline{\mathcal{F}^2(\mathbb{C})})^K$ with the Hilbert space pairing, $G(f,g) =
f(\zeta)\overline{g(\eta)}(1,...,1)$,
$\Phi = (\phi_1(\zeta)\overline{\psi_1(\eta)}, ...,\phi_K(\zeta)\overline{\psi_K(\eta)})$,
which gives
\begin{equation*}
 F_{G,\Phi}(f,g) = \sum_{j=1}^K \langle f, \phi_j \rangle \langle \psi_j, g \rangle.
\end{equation*}
In the special case, when $\phi_j = \kb_{\vs_j}$ and $\psi_j = \kb_{\z_j}$, for some points
$\vs_j$ and $\z_j$ in $\mathbb{C}$, we have $G(f,g) =
( f(\vs_1)\overline{g(\z_1)}, ... ,f(\vs_K)\overline{g(\z_K)} )$, so that
\begin{equation*}
 F_{G,\Phi_\mathbf{1}}(f,g) = \sum_{j=1}^K f(\vs_j)\overline{g(\z_j)}.
\end{equation*}
Another representation of this form can be obtained by setting \\
$X = (\Ec(\mathbb{C})\otimes \Ec(\mathbb{C}))^K$, $X' = (\Ec'(\mathbb{C})\otimes
\Ec'(\mathbb{C}))^K$,
\begin{eqnarray*}
 G(f,g) &=& \omega(\x) f(\x)\omega(\eta)\overline{g(\eta)}(1,...,1), \\
 \Phi &=& \omega^{-1}(\x)\omega^{-1}(\eta)(\delta(\x - \vs_1)\delta(\eta - \z_1), ...,\delta(\x
 - \vs_K)\delta(\eta - \z_K)),
\end{eqnarray*}
which gives
\begin{equation*}
 F_{G,\Phi}(f,g) = \sum_{j=1}^K f(\vs_j)\overline{g(\z_j)}.
\end{equation*}
\end{example}

\section{Estimates for derivatives  of functions \\ in the Fock space}\label{Sect.Prop.Fock}

In order to extend the set of admissible distributional symbols beyond
distributions with compact support we establish, in this section, some additional properties of functions in the Fock space $\Fc^2(\C)$.

\subsection{Fock-Carleson measures for derivatives; norm estimates}

It is well known, see, e.g., \cite{ZhuFock}, that functions in the Fock space satisfy the
growth estimate:
\begin{equation}\label{eq.Fock.Est}
    |f(z)|\le e^{|z|^2/2}\|f\|_{\Fc^2}, \ f\in\Fc^2.
\end{equation}
Certain estimate for  derivatives of $f\in\Fc^2(\C)$, generalizing \eqref{eq.Fock.Est},
has been established in \cite{PeTaVi2}:
\begin{equation*}
    \|(1+|z|)^{-k}f^{(k)}e^{-|z|^2}\|_{L_2}\le C_k\|f\|_{\Fc^2},
\end{equation*}
with some, non-specified constants $C_k$. We need, however,  pointwise estimates for the
derivatives for  $f$, with explicitly shown dependence
of  the constants on the order  $k$ of the differentiation.

\begin{proposition}\label{prop:est.deriv.1}
Let $f\in\Fc^2$. Then
\begin{equation}\label{eq:est.deriv.1}
    |f^{(k)}(z)|\le C k!(1+|z|)^k e^{|z|^2/2}\|f\|_{\Fc^2},
\end{equation}
with some constant $C$ not depending on $k$.
\end{proposition}
\begin{proof} For a given $z\in\C$, we fix some $s=s(z)$, to be determined later, and write the
Cauchy formula
\begin{equation}\label{eq:est.deriv.2}
    f^{(k)}(z)=k!(2\pi i)^{-1}\int_{|z-\z|=s}(z-\z)^{-k-1}f(\z)d\z.
\end{equation}
If $|z|\le 1$, we take $s=s(z)=1$, and, by \eqref{eq:est.deriv.2}, obtain
\begin{equation}\label{eq:est.deriv.3}
    |f^{(k)}(z)|\le k!\max_{|z|\le2}|f(z)|\le k!\|f\|_{\Fc^2}e^{2}.
\end{equation}

Now let $|z|>1$. We apply the Cauchy formula with $s=s(z)=|z|^{-1}$.
With this choice, \eqref{eq:est.deriv.2} gives
\begin{equation*}
    |f^{(k)}(z)|\le k! |z|^{k}\max_{|\z|\le |z|+|z|^{-1}}|f(\z)||z|^{k}.
\end{equation*}

We substitute here estimate \eqref{eq.Fock.Est}, which leads to the inequality we need:
\begin{equation*}
    |f^{(k)}(z)|\le C k! (1+|z|)^ke^{(|z|+|z|^{-1})^2/2}\|f\|_{\Fc^2}\le
    k!(1+|z|)^ke^{(|z|^2+1)/2}\|f\|_{\Fc^2}. \qedhere
\end{equation*}

\end{proof}

We complement now Proposition \ref{prop:est.deriv.1} by a sharper local estimate, where the
$\Fc^2$-norm on the right hand side in \eqref{eq:est.deriv.1} is replaced by an integral, involving $f$, over a certain disk.

\begin{proposition}\label{prop:est.deriv.2}
Let $f$ be an entire analytical function. For any $r>1$, with some constant $C(r)$, depending
only on $r$,

\begin{equation*}
    |f^{(k)}(z)|^2\le k!^2C(r)e^{|z|^2}(1+|z|)^{2k}\int_{B(z,r)}|f(\z)e^{-|\z|^2/2}|^2d V(\z)
\end{equation*}

\end{proposition}
\begin{proof} For small $|z|$, the estimate follows directly from the first of inequalities in
\eqref{eq:est.deriv.3} which, actually, does not
require $f\in\Fc^2(\C)$ and holds for any analytical function $f$.  So, let $|z|>1$. By
\eqref{eq:est.deriv.2},
\begin{equation}\label{eq:est.deriv.7}
   |f^{(k)}(z)e^{-|z^2|/2}|^2\le k!^2(1+|z|)^{2k}e^{-|z^2|}\max_{|w-z|\le s}|f(w)|^2,
\end{equation}
and, due to our choice of $s=|z|^{-1}$, \eqref{eq:est.deriv.7} leads to
\begin{equation}\label{eq:est.deriv.8}
    |f^{(k)}(z)e^{-|z^2|/2}|^2\le e^2 k!^2(1+|z|)^{2k}\max_{|w-z|\le |z|^{-1}}
    \left(e^{-|w|^2}|f(w)|^2.\right)
\end{equation}
Now we apply Lemma 2.1 in \cite{IsrZhu}, which gives an estimate of the quantity on the right
in \eqref{eq:est.deriv.8}:
\begin{equation}\label{eq:est.deriv.9}
    |f(w)e^{-|w|^2/2}|^2\le  C(r)\int_{B(w,r)}|f(\z)e^{-|\z|^2/2}|^2 dV(\z).
\end{equation}
For $|w-z|\le s$, the disk $B(w,r)$ lies inside the disk $B(z,r+s)$, therefore
\begin{equation}\label{eq:est.deriv.10}
    |f^{(k)}(z)e^{-|z^2|/2}|^2\le e^2 C(r)  k!^2\int_{B(z,r+s)} |f(\z)e^{-|\z|^2/2}|^2 dV(\z). \qedhere
\end{equation}
\end{proof}

The results above lead to the notion of Fock-Carleson measures for derivatives.

\begin{definition}A complex valued measure $\m$ on $\C$ is called Fock-Carleson measure for
derivatives of order $k$ ($k$-FC measure, in short),
if, for some constant $\varpi_k(\m)$, for any function $f\in \Fc^2(\C)$,  the following
inequality holds
\begin{equation}\label{eq:Carles.k}
    \int_{\mathbb{C}} |f^{(k)}(z)|^2 e^{-|z|^2} d|\m|(z)\le \varpi_k(\m)\|f\|_{\Fc^2}^2,
\end{equation}
where, recall, $|\m|$ denotes the variation of the measure $\m.$
\end{definition}

An explicit  description of such measures follows immediately from Proposition
\ref{prop:est.deriv.1}:

\begin{theorem}\label{Th.Carleson}A measure $\m$ is a  $k$-FC measure if and only if, for some
(and, therefore, for any) $r>0$, the quantity
\begin{equation}\label{eq:Carleson.Constant}
    C_k(\m,r)=(k!)^2\sup_{z\in\C}\left\{|\m|(B(z,r))(1+|z|^2)^{k}\right\}
\end{equation}
is finite. For a fixed $r$,  the constant $\varpi_k(\m)$ in \eqref{eq:Carles.k} can be taken as
$\varpi_k(\m)=C(r)C_k(\m,r)$, with some
coefficient $C(r)$ depending only on $r$.
\end{theorem}

In other words, Theorem \ref{Th.Carleson} states that $\m$ is a $k$-FC measure if and only if
$(1+|z|^2)^{k}\m$ is a FC measure. If $k=0$, then any
$0$--FC measure is just a FC measure.

Further on, the parameter $r$ will be fixed (say, being equal to $\sqrt{2}$), and the
dependence on $r$ will be suppressed in the notations. The
quantity $\varpi_k(\m)$ in \eqref{eq:Carles.k} will be called the $k$-FC norm of the measure
$\m$.

\begin{proof}The proof of the sufficiency follows mainly the one for  Fock-Carleson (0-FC) measures in
\cite{IsrZhu}, with the replacement of Lemma 2.1
there by our Proposition \ref{prop:est.deriv.1}. So, suppose that \eqref{eq:Carleson.Constant}
is satisfied, i.e.,
 \begin{equation}\label{eq:Carleson.C1}
    |\m|(B(z))(1+|z|^2)^{k}\le C_k
 \end{equation}
 for all disks $B(z)$ (with radius $\sqrt{2}$).

 Consider the lattice $\Z^2\subset\C^1$ consisting of points with both real and imaginary parts
 being  integer. The disks $B_\nb=B(n_1+in_2)$,
 $\nb=n_1+in_2\in \Z^2$, form a covering of $\C$ with multiplicity 9. Denoting the left-hand
 side in \eqref{eq:Carles.k} by $I_k(f)$, we have
 \begin{equation*}
    I_k(f)\le \sum_{\nb\in\Z^2}\int_{B_\nb}|f^{(k)}(w)|^2e^{-|w^2|}d|\m|(w.)
 \end{equation*}
 By   Proposition \ref{prop:est.deriv.1} and the triangle inequality, there exists a constant
 $M$, not depending on $k$,  such that
 \begin{equation*}
    |f^{k}(w)|^2e^{-|w^2|}\le M
    k!^2(1+|w|)^{2k}\int_{\widetilde{B_{\nb}}}|f(\z)|^2e^{-|\z|^2}dV(\z),
 \end{equation*}
 for all $w\in B_\nb$, where $\widetilde{B_{\nb}}$ denotes the concentric with $B_{\nb}$ disk,
 twice as large. Thus, by \eqref{eq:Carleson.C1},
 \begin{equation*}
    I_k(f)\le M C_k\sum_{\nb\in\Z^2}\int_{\widetilde{B_{\nb}}}|f(u)|^2e^{-|u|^2}dA(u).
 \end{equation*}
 Now, since the disks $\widetilde{B_{\nb}}$ form a covering of $\C$ with finite multiplicity,
 the last inequality implies the required estimate
 for $I_k(f)$ via the Fock norm of the function $f$. The necessity part of the theorem is
 proved exactly as in \cite{IsrZhu}, by means of evaluating
 $I_k(f)$ on the normalized reproducing kernels.
\end{proof}

It is convenient to extend the notion of $k$-FC measures to half-integer values of $k$, defining these measures as those for which
the relation \eqref{eq:Carleson.C1} holds.
\begin{corollary}\label{5.5}
 For any  integers $  p \in \mathbb{Z}_+$ and integer or half-integer $k$ a measure $\mu$ is a $k$-FC measure if and
 only if the measure $\mu_p = (1 + |z|^2)^{(k-p)} \mu$
 is a $p$-FC measure, moreover, for integer $k$ .
 \begin{equation}\label{eq.eqiuvalence}
 C_p(\mu,r)\asymp C_k(\m_p,r)
 \end{equation}
\end{corollary}

Note, in particular, that any measure with compact support is a $k$-FC measure for any $k.$

\begin{remark}Estimates similar to the ones derived in this section have been obtained in
\cite{PeTaVi2}, however the results there are weaker
than ours: the inequality similar to \eqref{eq:Carles.k} was proved only under the condition
that  the measure $\m$ has the form $(1+|z|)^{-2k}\f(z)dV(z)$,
with locally bounded function $\f$, moreover, the dependence of the constant on $k$ was not
explicitly  found. Our results can also be considered as a kind of dual
to the estimates for Fock-Sobolev spaces in \cite{ChoZhu}.
\end{remark}

\subsection{Sharper estimates}
We finish this section by obtaining a sharper form of the the above estimates, which provide a
better constant for large values of $|z|$.

\begin{proposition}\label{prop:estimates.better}
Let $f$ be an analytical function if $\C$, $k\in\Z_+$. We denote by $\g(k,z)$ the quantity
\begin{equation}\label{eq:est.better.1}
        \g(k,z)=\left\{\begin{array}{c}
                        k!, \ \ \textrm{for } |z|\le k, \\
                        k^{-\frac12}, \ \ \textrm{for } |z|>k.
                      \end{array}\right.
\end{equation}
Then
\begin{equation}\label{eq:est.better.2}
    |f^{(k)}(z)|\le C \g(k,z) (1+|z|)^k e^{|z|^2/2}\|f\|_{\Fc^2},
\end{equation}
for a constant $C$ not depending on $k$.
Thus, compared with \eqref{eq:est.deriv.1}, the constant in the estimate for the growth of
derivatives of $f$ is considerably improved.
\end{proposition}

\begin{proof}
The proof follows the one of Proposition \ref{prop:est.deriv.1}; the improvement is achieved by
a better choice of $s$ in \eqref{eq:est.deriv.2}
for $|z|\ge k$: we take here  $s=k/|z|\le 1$. Substitution in \eqref{eq:est.deriv.2} gives
\begin{eqnarray}\label{eq:est.better.3} \nonumber
    |f^{(k)}(z)| &\le& k! k^{-k}|z|^ke^{(|z|+k|z|^{-1})^2/2}\|f\|_{\Fc^2} \\
    &\le& C k!(k/e)^{-k}|z|^ke^{(|z|^2)/2}\|f\|_{\Fc^2}.
\end{eqnarray}
By the Stirling formula, the quantity $k!(k/e)^{-k}$ is majorated by $Ck^{-\frac12}$, which
proves the required estimate.
\end{proof}
In a similar way, Proposition \ref{prop:est.deriv.2} and Theorem \ref{Th.Carleson} can be
sharpened. We give here only the formulations; the proofs follow the proofs of Proposition
\ref{prop:est.deriv.2} and Theorem \ref{Th.Carleson}, only with the choice $s=|z|^{-1}$
replaced by $s=k|z|^{-1}$, for $|z|\ge k$.
\begin{proposition}\label{prop:est.deriv.better.2}Let $f$ be an entire analytical function. For any $r>1$, with some constant $C(r)$ depending
only on $r$,
\begin{equation}\label{eq:est.deriv.6.better}
  |f^{(k)}(z)|^2\le \g(k,z)^2e^{|z|^2}(1+|z|)^{2k}\int_{B(z,r)}|f(\z)e^{-|\z|^2/2}|^2
dV(\z).\end{equation}
\end{proposition}
\begin{theorem}\label{Th.Carleson.better}A measure $\m$ is a $k$-FC if and only if, for some
(and therefore for any) $r>0$, the quantity
\begin{equation}\label{eq:FC improved}
  \tilde{C}_k(\m,r)=\sup_{z\in\C}\{|\m|(B(z,r))\g(k,z)(1+|z|^2)^k\}
\end{equation}
is finite. For a fixed $r$, the constant $\varpi_k(\m)$ in \eqref{eq:Carles.k} can be chosen as
$C(r)\tilde{C}_k(\m,r).$
\end{theorem}
It is important to remark that the finiteness of the expression \eqref{eq:FC improved} is
equivalent to the finiteness of \eqref{eq:Carleson.Constant}, since the behavior of the
functions in these formulas as $|z|\to\infty$ is the same. However, if the support of measure
$\m$ lies far away from the origin, in the domain $|z|>k$, Theorem \ref{Th.Carleson.better} can
give a better value of the constant in inequality \eqref{eq:Carles.k}.

\section{Sesquilinear forms of infinite differential order \\ and corresponding Toeplitz
operators}
Having the estimates of Fock functions and their derivatives at hand, we can introduce
  sesquilinear forms involving derivatives, first  of a finite, and then of the  infinite
 order.
 \subsection{Forms of a finite order}
 \begin{proposition}\label{prop:integral}
 Let $\m$ be a $k$-FC measure, with integer or half-integer $k$. With $\m$ we associate the sesquilinear form
\begin{equation}\label{eq:C-form.1}
    F(f,g)=\int_{\C}f^{(\a)}(z)\overline{g^{(\b)}(z)}e^{-|z|^2}d\mu(z),\ f,g\in\Fc^2(\C),
\end{equation}
for some $\a,\b$ with $\a+\b=2k$.
This form is bounded in $\Fc^2(\C)$, moreover
\begin{equation}\label{eq:C-form.1a}
   |F(f,g)|\le C(F)\|f\|_{\Fc^2}\|g\|_{\Fc^2}, \ \mathrm{with} \ C(F)\le (\varpi_\a(\m)\varpi_\b(\m))^{\frac12}.
\end{equation}
\end{proposition}
\begin{proof}
Since $\a+\b=2k$, by the Cauchy inequality, we have
\begin{eqnarray}\label{eq:C-form.2} \nonumber
 |F(f,g)| &\le&
 \left(\int_{\C}|f^{(\a)}(z)|^2e^{-|z|^2}(1+|z|^2)^{k-\a}d|\mu|(z)\right)^{\frac12} \\
&\times& \left(\int_{\C}|g^{(\b)}(z)|^2e^{-|z|^2}(1+|z|^2)^{k-\b} d|\mu|(z)\right)^{\frac12}.
\end{eqnarray}
By Corollary \ref{5.5}, $\m_{\a}=(1+|z|^2)^{k-\a}\m$ is an $\a$-FC measure, $\m_\b=(1+|z|^2)^{k-\b}\m$ is a
$\b$-FC measure, and thus we arrive at the required
estimate $|F(f,g)|\le C(F)\|f\|_{\Fc^2}\|g\|_{\Fc^2}$, with a constant $C(F)\le \varpi_k(\m)$.
\end{proof}

As usual, for any norm estimate for the operator defined by a symbol, the boundedness result is
accompanied by a compactness result.

\begin{definition}\label{def:vanishing kCarleson} A measure $\m$ is called \emph{vanishing
$k$-FC measure} if
\begin{equation*}
    \lim_{|z|\to\infty}(|\m|(B(z,r))(1+|z|^2)^k)=0.
\end{equation*}
\end{definition}
\begin{corollary}\label{Cor:vanishing} Let $\m$ be a vanishing $k$-FC measure, with integer or half-integer $k$. Then the
operator in $\Fc^2(\C)$ defined by the form
\eqref{eq:C-form.1}, with $\a+\b=2k$ is compact.
\end{corollary}
\begin{proof} It goes by the usual pattern, see, e.g., \cite[Theorem 2.4]{IsrZhu}. For a given $\varepsilon>0$, we find  $R$ so large
that $|\m|(B(z,r))(1+|z|^2)^k\le \varepsilon$ for
$|z|>R$. We split the measure $\m$ into the sum $\m=\m_R+\m_R'$, with $\m_R$ supported in the
disk $|z|\le R$, and $\m_R'$ outside this disk, so
that $\varpi_k(\mu_R')\le \varepsilon $. This splitting leads to the corresponding  splitting
of the sesquilinear form: $F=F_R+F_R'$. The form
$F_R$, corresponding to the symbol with compact support, generates a compact operator. On the other hand, the estimate \eqref{eq:C-form.1a},
applied to $F_R'$, implies the smallness of the operator defined by this sesquilinear form.
These two observations prove the compactness of the
operator in question.
\end{proof}
Note, again, that any measure with compact support is a vanishing $k$-FC measure for any $k$, and therefore, for
any $\a,\b$ the operator in $\Fc^2(\C)$ generated by the form \eqref{eq:C-form.1} with such
measure, is compact.

Now we introduce a class of distributions generated by derivatives  of FC measures.

Here and further on, we will simultaneously use the scalar product in the weighted $L_2$-space, and the action of the related distribution,
where
the action, traditionally, is induced by the (bilinear!) pairing with respect to the
Lebesgue measure. Recall that the scalar product in a
Hilbert space is always denoted by $\langle \cdot,\cdot \rangle$, while the action of the
distribution on a smooth function is always denoted by
$( \cdot,\cdot)$.

Two related distributions will be denoted by the same alphabetic symbol, one of them regular,
the other one boldfaced, say, $\F$ and $\Fp$. The
relation between them is given by
 \begin{equation*}
    \Fp=\o\F, \ \ \mathrm{with} \ \  \o(z)=\pi^{-1}e^{-|z|^2},
 \end{equation*}
so that, in the case when $\F$ is a function,
 \begin{equation*}
    ( \Fp,h) = \langle \F,\bar{h} \rangle.
 \end{equation*}

Let $\F$ be a distribution in $\Dc'(\C)$. By \emph{coderivative} of $\F$ we mean the distribution
 \begin{equation}\label{eq:RelDistr3}
    \pp\F=\o(z)^{-1}\partial{(\o(z)\F)}=\o(z)^{-1}\partial(\Fp),
 \end{equation}
thus,
\begin{equation}\label{eq:RelDistr.4}
    \langle \pp\F,h\rangle=( \partial\Fp,\bar{h}),
\end{equation}
as soon as the the expressions in \eqref{eq:RelDistr.4} make sense.

Let $\m$ be a $k$-FC measure, and $\a+\b=2k$.
The coderivative $\pp^\a\overline{\pp}^\b\m$ is defined in accordance with
\eqref{eq:RelDistr3},
so that, for a function $h(z)=f(z)\bar{g}(z)$, $f,g\in \Fc^2$,

\begin{equation*}
   (\pp^\a\overline{\pp}^\b\m,h) = (-1)^{\a+\b}( \omega \m, \partial^\a\bar{\partial}^\b
   {h})=(-1)^{\a+\b}(\omega\m,\partial^\a f\overline{\partial^\b g}),
\end{equation*}
again, provided the right-hand side makes sense.

\begin{proposition}\label{prop:Toepl.Distr} If  $\m$ is a $k$-FC measure, and
$\a+\b=2k$,
then the sesquilinear form in $\Fc^2(\C)$
\begin{equation}\label{eq:Toepl.Distr1}
    F(f,g)=F_{\m,\a,\b}(f,g)=( \pp^\a\overline{\pp}^\b\m,f\bar{g})
\end{equation}
is bounded in $\Fc^2(\C)$,
\begin{equation}\label{eq:Toepl.Distr2}
    |F(f,g)|\le C (\varpi_a(\m)\varpi_\b(\m))^{\frac12}\|f\|_{\Fc^2}\|g\|_{\Fc^2}.
\end{equation}
\end{proposition}
\begin{proof}The proof follows immediately from the definition of the sesquilinear form and
Proposition \ref{prop:integral}.
\end{proof}

We show now how the form (\ref{eq:Toepl.Distr1}) fits  the construction of
Section~\ref{se:construction}.

\begin{example}\label{ex:OperatorsDefinedByDerivatives}\textsf{Toeplitz operators defined by
the derivatives of $k$-FC measures.} \\
The sesquilinear form that corresponds to the derivative $\pp^\a\overline{\pp}^\b\m$ of a
$k$-FC measure $\mu$ admits the following
description: $X = L_1(\mathbb{C}, d\mu)$, $X' = L_{\infty}(\mathbb{C}, d\mu)$,  $G(f,g)(z) =
f^{(\alpha)}(z)\overline{g^{(\beta)}(z)}\omega(z)$,
$\Phi= (-1)^{\alpha+\beta} \in L_{\infty}(\mathbb{C}, d\mu)$, so that
\begin{equation} \label{eq:Form_derivatives}
 F_{G,\Phi}(f,g) = \int_{\mathbb{C}} (-1)^{\alpha+\beta} f^{(\alpha)}(z)
 \overline{g^{(\beta)}(z)}\omega(z)d\mu.
\end{equation}

and
\begin{equation}\label{eq:OperDefByDeriv}
 (T_{F_{G,\Phi}}f)(z) = F_{G,\Phi}(f,k_z) = \frac{(-1)^{\alpha+\beta}}{\pi} \int_{\mathbb{C}}
 z^{\beta}
 f^{(\alpha)}(w)e^{(z-w)\overline{w}}d\mu(w).
\end{equation}
\end{example}

We give now illustrative examples to the above notions.

\begin{example}\label{ex:Toepl.k-Carleson}
Consider the measure $\m$ supported in the integer lattice $\Z^2\subset\R^2=\C^1$: $\Z^2$
consists of points with both co-ordinates integer.
Suppose that the measure $\mu $ of the node $\nb =(n_1,n_2)= n_1+in_2$ of the lattice
satisfies the condition $|\m(\nb)|\le
C(|n_1|+|n_2|)^{-2k}$. Then, due to Theorem \ref{Th.Carleson}, $\m$ is a \emph{k-CF} measure,
and,  for $\a+\b= 2k$, the Toeplitz operator
$T_{\pp^{\a}\bar{\pp}^{\b}\m}$ is bounded. By \eqref{eq:OperDefByDeriv}, this operator acts as
\begin{equation*}
    (T_{\pp^{\a}\bar{\pp}^{\b}\m}f)(z)=\frac{(-1)^k}{\pi}z^{\b}\sum_{\nb\in
    \Z^2}f^{(\a)}(\nb)e^{(z-\nb)\overline{\nb}}\mu(\nb).
\end{equation*}

\end{example}

\begin{example}
Given $k,\, \alpha,\, \beta \in \mathbb{Z}_+$, we introduce the $k$-FC measure
\begin{equation*}
 d\mu_k = (1+|z|^2)^{-k}\,dV(z),
\end{equation*}
and the corresponding form (\ref{eq:Form_derivatives})
\begin{equation*}
 F_{\alpha,\beta,k} (f,g) = (-1)^{\alpha+\beta} \int_{\mathbb{C}}  f^{(\alpha)}(z)
 \overline{g^{(\beta)}(z)}\, \frac{e^{-|z|^2}}{\pi
 (1+|z|^2)^k}\, dV(z).
\end{equation*}
Direct calculations give
\begin{eqnarray*}
 F_{\alpha,\beta,k}(e_n,e_m) &=& (-1)^{\alpha+\beta} \int_{\mathbb{C}}  e_n^{(\alpha)}(z)
 \overline{e_m^{(\beta)}(z)}\, \frac{e^{-|z|^2}}{\pi
 (1+|z|^2)^k}\, dV(z) \\
&=&  \frac{(-1)^{\alpha+\beta}\sqrt{n! m!}}{(n - \alpha)! (m-\beta)!} \int_{\mathbb{C}}
 z^{n-\alpha} \overline{z}^{m-\beta} \frac{e^{-|z|^2}}{\pi (1+|z|^2)^k}\, dV(z) \\
&=& \frac{(-1)^{\alpha+\beta}\sqrt{n! m!}}{(n - \alpha)! (m-\beta)!} \int_{\mathbb{R}_+}
\frac{r^{(n-\alpha)+(m-\beta)}e^{-r^2}}{(1+r^2)^k}rdr \\
&\times& \int_0^{2\pi} e^{i\theta[(n-\alpha)+(m-\beta)]}\frac{d\theta}{\pi}.
\end{eqnarray*}
This expression is always equal to $0$ if $n < \alpha$ or $n-\alpha \neq m-\beta$.
Taking now $n \geq \alpha$ and $n-\alpha = m-\beta$, and thus $m = n -\alpha +\beta$, we have
\begin{eqnarray*}
 F_{\alpha,\beta,k}(e_n,e_{n -\alpha +\beta}) &=& \frac{(-1)^{\alpha+\beta}\sqrt{n! (n -\alpha
 +\beta)!}}{[(n - \alpha)!]^2} \int_{\mathbb{R}_+}
\frac{r^{2(n-\alpha)}e^{-r^2}}{(1+r^2)^k}2rdr \\
&=& \frac{(-1)^{\alpha+\beta}\sqrt{n! (n -\alpha +\beta)!}}{[(n - \alpha)!]^2}
\int_{\mathbb{R}_+}
\frac{s^{n-\alpha}e^{-s}}{(1+s)^k}ds\\
&=& \gamma_{\alpha,\beta,k}(n).
\end{eqnarray*}
In this way, the form $F_{\alpha,\beta,k}$ defines via (\ref{eq:OperDefByDeriv}) a densely
defined (unbounded, in general) Toeplitz operator
$T_{\pp^{\a}\bar{\pp}^{\b}\m_k}$, whose domain contains all standard basis elements $e_n(z)$,
$n \in \mathbb{Z}_+$, and
\begin{equation*}
 T_{\pp^{\a}\bar{\pp}^{\b}\m_k} e_n =
\left\{\begin{array}{ll}
   \gamma_{\alpha,\beta,k}(n)\, e_{n -\alpha +\beta}, & \textrm{if} \ \ n \geq \alpha, \\
    0, & \textrm{otherwise}.
\end{array}\right.
\end{equation*}
For generic $k,\, \alpha,\, \beta \in \mathbb{Z}_+$ the exact formula for
$\gamma_{\alpha,\beta,k}(n)$ is rather complicated, but its asymptotic
behavior for large $n$ is quite simple. For $n > \alpha +k$, we have
\begin{equation*}
 \gamma_{\alpha,\beta,k}(n) = (-1)^{\alpha+\beta} \frac{\sqrt{n! (n -\alpha
 +\beta)!}\,(n-\alpha -k)!}{[(n - \alpha)!]^2}\, \left( 1 +
 O\left(\frac{1}{n}\right)\right),
\end{equation*}
or, by \cite[Formula 8.328.2]{GradshteynRyzhik80},
\begin{equation*}
 \gamma_{\alpha,\beta,k}(n) = (-1)^{\alpha+\beta}
 \frac{(n-\alpha)^{\frac{\alpha+\beta}{2}}}{(n-\alpha-k)^k}\, \left( 1 +
 O\left(\frac{1}{n}\right)\right).
\end{equation*}
This shows that the Toeplitz operator $T_{\pp^{\a}\bar{\pp}^{\b}\m_k}$ is bounded if and only if
$\alpha+\beta \leq 2k$; if $\alpha+\beta < 2k$ the
operator is even compact. If $\alpha = \beta = k$ the operator $T_{\pp^{\a}\bar{\pp}^{\b}\m_k}$
is a compact perturbation of the identity operator~$I$.
\end{example}

\begin{remark}
 {\rm
It is interesting to observe that for the specific cases $\alpha = 0, \ \beta = 1, \ k =0$
and $\alpha = 1, \ \beta = 0, \ k =0$ (the operators are unbounded), the corresponding Toeplitz
operators, considered on the natural domain, coincide with the
classical creation and annihilation operators in the Fock space,
\begin{equation*}
 \mathfrak{a}^{\dag} = zI, \ \ \ \ \ \mathfrak{a} = \frac{\partial }{\partial z},
\end{equation*}
respectively. }
\end{remark}

We will call the derivatives of $k$-FC measures and the corresponding sesquilinear forms
`\emph{symbols of finite type}'.

\subsection{Symbols of weak almost-finite type}
In this subsection we introduce a class of symbols, more general than the derivatives of $k$-FC
measures, defined previously. They might be
considered as a kind of hyperfunctions, however, this would be not quite consequential, so we
use the term `\emph{symbols of weak almost-finite
type}' instead.

\begin{definition}\label{def:almost finite} Let $F(f,g)$ be a bounded sesquilinear form on the
Fock space $\Fc^2(\C)$. We say that this form is a symbol of weak almost-finite type if there
exist a collection $\pmb{\m}=\{\m_{\a,\b}\}_{\a,\b=0,1,2,\dots}$ of $(\a+\b)/2$-FC
measures such that, for each $f,g\in \Fc^2$, the series
\begin{equation}\label{eq:almost finite.1}
    \sum_j\sum_{\max(\a,\b)=j}F_{\a,\b}(f,g)\equiv \sum_j\sum_{\max(\a,\b)=j}(
    \pp^{\a}\overline{\pp}^\b\m_{\a,\b}, f\bar{g}) ,
\end{equation}
 converges to $F(f,g)$. .
\end{definition}
It is convenient to consider $\pmb{\m}$ as the formal sum
\begin{equation*}
  \pmb{\m}\simeq \sum_{\a,\b} \pp^{\a}\overline{\pp}^\b\m_{\a,\b}.
\end{equation*}

The Banach-Steinhaus theorem (more exactly, its version for bi-linear mappings, see, e.g.,
Sect. 7.7 in \cite{Edv}) implies that the condition of boundedness of the sesquilinear form
$F(f,g)$ in this definition is superfluous, it
follows automatically from \eqref{eq:almost finite.1}.

The condition \eqref{eq:almost finite.1} can be in an obvious way  formulated in another form,
involving Toeplitz operators corresponding to
$F_j=\sum_{\max(\a,\b)=j} F_{\a,\b}$.

\begin{proposition}\label{prop:weak conv.1}Let $F(f,g)$ be a symbol of weak almost-finite type.
Then the sequence of Toeplitz operators $T_j=\sum_{s\le j}T_{F_s}$
converges weakly to the Toeplitz operator $T_F.$
\end{proposition}

It follows, in particular, that for symbols of weak almost-finite type, the point-wise
convergence takes place:
\begin{proposition}\label{prop:pointwise1} If $F$ is a symbol of weak almost-finite type, then,
for any $f\in\Fc^2(\C)$ and any $z\in\C$, the
sequence $(T_jf)(z)$ converges to $(T_{F}f)(z)$.
\end{proposition}
The proof follows immediately from the relation $(T_F f)(z)=F(f,\kb_z)$ and similar relations
for $F_j$.

Moreover, the partially converse statement is correct.

\begin{proposition} If the operators $T_j$ are uniformly bounded and the sequence $T_j f$
converges to $Tf$ point-wise: $(T_jf)(z)$ converges to $(T_{F}f)(z)$ for all $f\in \Fc^2$ and $z\in \C$, then $T_j $ converges to $T$ weakly.
\end{proposition}
 The (rather standard) proof goes the following way. A given $g\in \Fc$ satisfies $g(z)=\int
 \kb_z(w) g(w)\o(w) dw$, therefore $g$ can be norm approximated, with error norm less than $\varepsilon$, by a finite linear combination
 $g_\varepsilon$ of $\kb_{z_m}$. On such combinations the convergence of $(T_j f,g_\varepsilon)$
 follows by linearity, and the possibility of passing to the limit follows from the uniform
 boundedness.

As the following theorem shows, the notion we have just introduced is too general.
\begin{theorem}\label{th:weak-universal} Any bounded operator in $\mathcal{F}^2(\mathbb{C})$ is
an operator with symbol of weak almost finite
type.
\end{theorem}
\begin{proof} Let $T$ be a bounded operator  in $\Fc^2(\mathbb{C})$. We consider the
representation of the operator $T$ as an infinite matrix in
the orthogonal basis $e_\a$ defined in \eqref{eq:basis}:
\begin{equation}\label{eq:af.1}
    T=\sum_{\a,\b} T_{\a\b}=\sum_{\a,\b}P_\a TP_\b,
\end{equation}
where $P_\a=P_{\a\a}$ is the orthogonal projection onto the one-dimensional subspace spanned by
$e_\a$. As described in Example 2.4, each
operator $T_{\a\b}=P_\a TP_\b$, being a rank one operator, with the source
 subspace spanned by $e_\b$ and the range spanned by $e_\a$, is in
fact the operator $\s_{\a\b} P_{\a\b}$, with  numerical coefficient $\s_{\a\b}=(Te_\a,e_\b)$
and the rank one operators $P_{\a\b}$ described in
\eqref{Action4}. Now, by \eqref{Action2}, \eqref{Action3}, the operator $\s_{\a\b} P_{\a\b}$ is
the Toeplitz operator associated with the
distributional symbol $\s_{\a\b}\F_{\a\b}\in \Ec'(\C)$. Moreover, any  such symbol is, by
\eqref{Struct2}, a collection of derivatives of continuous functions with compact support, and
the latter correspond to $k$-FC measures for any $k$. Therefore, any symbol
$\s_{\a\b}\F_{\a\b}$ is the sum of derivatives of order not greater than $\a+\b$
of $k$-FC measures. Finally, we establish the convergence of the sesquilinear forms, required
by the theorem. Denote by $\Pc_j$ the projection
  $\Pc_j=
 \sum_{\a\le j}P_\a$. Thus,
 the operator
\begin{equation*}
    \Pc_j T \Pc_j=\sum_{\a,\b\le j}T_{\a\b}
\end{equation*}
is the Toeplitz operator with symbol being the sum of coderivatives of $j$-FC measures. On the
other hand, the projections $\Pc_j$ converge
strongly to the identity operator as $j\to\infty$, therefore for any $f,g\in
\Fc^2(\mathbb{C})$, we have
\begin{equation*}
    \langle\Pc_j T\Pc_j f,g\rangle=\langle T\Pc_j f,\Pc_j g\rangle\to \langle T f, g\rangle. \qedhere
\end{equation*}
\end{proof}

\begin{remark}
The theorem  implies, in particular, that given any bounded linear operator on the Fock space, any finite truncation of the infinite
matrix representation with respect to the standard basis in $\mathcal{F}^2(\mathbb{C})$ is a Toeplitz operator with symbol being the sum of
derivatives of certain $k$-FC measures.
\end{remark}

Consider now a special case of the above theorem when the operator $T$ is compact. Then the
sequence $\{\Pc_j T\Pc_j\}_{k \in \mathbb{Z}_+}$ converges to $T$ in norm,
\begin{equation*}
 \lim_{j \to \infty} \Pc_j T\Pc_j = T.
\end{equation*}
The sesquilinear form of the operator $T$ is the (norm) limit of the sequence of forms of
Toeplitz operators $\Pc_j T\Pc_j$, each one of which, as it was stated in the proof of Theorem
\ref{th:weak-universal}, is a sum of derivatives of order not greater than $\a+\b$
of $k$-FC measures. Note that the differential order of these forms may tend to infinity as $k
\to \infty$. Thus the limiting form will have, in general, an infinite differential order. We
will identify the sesquilinear form of the operator $T$ with the sequence
$\pmb{\m}=\{\m_{j}\}_{j \in \mathbb{Z}_+}$, where each $\m_{}$ is the form of the finite
rank Toeplitz operator $\Pc_j T\Pc_j$. Thus we come to the following important result.

\begin{theorem}
 Each compact operator on $\mathcal{F}^2(\mathbb{C})$ is a Toeplitz operator defined, in
 general, by a sesquilinear form of an infinite differential order.
\end{theorem}

\subsection{Symbols of norm almost-finite type}
Let $\pmb{\m}=\{\m_{\a,\b}\}_{\a,\b=0,1,2,\dots}$ be  a collection of $(\a+\b)/2$-FC
measures. We say that this collection is a \emph{ symbol of
norm almost finite type} if
\begin{equation}\label{eq.norm AF1}
|||\pmb{\m}|||=\sum_{\a,\b}\varpi_{\frac{\a+\b}2}(\m_{\a,\b})<\infty.
\end{equation}
\begin{theorem}Let $\pmb{\m}$ be a symbol of norm-almost finite type. Denote by
$T_{\a,\b}(\pmb{\m})$ the operator defined by the sesquilinear form
$F_{\a,\b}(f,g)=(\pp^{\a}\overline{\pp}^\b\m_{\a,\b},f\overline g)$, as in
\eqref{eq:Toepl.Distr1}.
Then the operator series $\sum_{\a,\b}T_{\a,\b}(\pmb{\m})$ converges in the norm sense.
\end{theorem}
The sum \begin{equation}\label{6.33}
  T(\pmb{\m})=\sum_{\a,\b}T_{\a,\b}(\pmb{\m})
\end{equation}
will be called the Toeplitz operator with symbol $\pmb{\m}$, which we associate with the
formal sum,
\begin{equation}\label{6.34}
  \pmb{\m}\simeq\sum_{\a,\b}\pp^\a\overline\pp^\b\m_{\a,\b}.
\end{equation}
\begin{proof}
By \eqref{eq:Toepl.Distr2}, the norm of the operator $T_{\a,\b}(\pmb{\m})$ is estimated
by $\varpi_{\frac{\a+\b}2}(\m_{\a,\b})$. Thus, the convergence of the series in \eqref{eq.norm
AF1} implies the norm convergence of the operator series.
\end{proof}
Note that if the measures $\m_{\a,\b}$ are vanishing $k$-CF measures then each of the
operators
$T_{\a,\b}(\pmb{\m})$ is compact and the norm-convergent series of such operators has
automatically a compact sum. It stands to reason that this remark concerns in particular the
case when all measures $\m_{\a,\b}$ have compact support.

Now we present several examples of the above symbols.
\begin{example}\label{Ex.normaf.1}Let $h_{\a,\b}(z)$, $\a,\b=0,1,\dots,$ be  bounded functions
in $\C$ such that
\begin{equation}\label{normaf.1.1}
  \sup_{z\in\C}\{|h_{\a,\b}(z)|(1+|z|)^{\a+\b}\}\le (\a!\b!)^{-1}c_{\a,\b}
\end{equation}
with $\sum c_{\a,\b}<\infty$. With each function $h_{\a,\b}(z)$ we associate  the measure
$\m_{\a,\b}$ that has the density $h_{\a,\b}(z)$ with respect to the Lebesgue measure,
$\m_{\a,\b}(E)=\int_E h_{\a,\b}(z)dV(z)$, and consider the symbol
\begin{equation}\label{normaf.2}
  \pmb{\m}\simeq\sum_{\a,\b} \pp^\a\overline\pp^\b{\m}_{\a,\b},
\end{equation}
 By \eqref{eq:Toepl.Distr1},
\eqref{eq:Toepl.Distr2}, the convergence condition \eqref{eq.norm AF1} is satisfied, and
therefore, the formal series \eqref{normaf.2} can serve as a symbol of a bounded Toeplitz
operator $T{\pmb{\m}}$.
\end{example}
\begin{example}\label{Ex.normaf.2} Here we modify the previous example in order to handle
discrete measures and their derivatives; at the same time we make use of the sharp norm
estimate \eqref{eq:FC improved}. Let $\d_\nb$, $\nb=n_1+in_2$, be the delta-measure placed at
the point $\nb$. With each point $\nb$ we associate the distribution
\begin{equation}\label{normaf.3}
  \theta_{\nb} =\pp^{\a_\nb}\overline\pp^{\b_{\nb}}{\d_\nb},
\end{equation}
with some $\a_\nb$, $\b_\nb$. Being a compactly supported distribution, $\theta_{\nb}$ defines
a compact Toeplitz $T_{\nb}$ operator in $\Fc^2$.
By Theorem \ref{Th.Carleson}, the norm of this Toeplitz operator is not greater than $p_{\nb}=C
\a_{\nb}!\b_{\nb}!(1+|\nb|)^{\a_{\nb}+\b_{\nb}}$, with some absolute constant $C$. Thus, the
formal sum  $\pmb{\m}=\sum A_\nb \theta_{\nb}$ of these distributions taken with
coefficients $A_{nb}$ defines a norm convergent series of Toeplitz operators $\sum T_{\nb}$ as
soon as $\sum |A_{\nb}|p_{\nb}<\infty$. So, we obtain a bounded Toeplitz operator corresponding
to a distributional symbol involving derivatives of unbounded order and supported in the
integer lattice.

Suppose now that the orders of derivatives  $\pp^{\a_\nb}\overline\pp^{\b_{\nb}}$ of
$\delta$-measures are sufficiently large,
\begin{equation}\label{normaf.4}
{\a_\nb},{\b_{\nb}}>|\nb|.
\end{equation}
Then, for estimating the norm the operator with symbol \eqref{normaf.3} we can use Theorem
\ref{Th.Carleson.better}. Under the condition \eqref{normaf.4}, we have
\begin{equation}\label{normaf.5}
\|T_{\nb}\|\le C q_{\nb}\equiv C(1+|\nb|)^{\a_{\nb}+\b_{\nb}}({\a_\nb}{\b_\nb})^{-\frac12}.
\end{equation}
Therefore, as soon as  $\sum |A_{\nb}|q_{\nb}<\infty$ converges, the operator series $\sum
T_{\nb}$ converges in the norm, which produces the Toeplitz operator. Note, that the latter
convergence condition is much weaker than the former one.
\end{example}

We conclude with an example of operators with symbol supported at one point.

\begin{example}\label{Ex.normaf.3} Fix a point $z_0=x_0+iy_0 \in \C$. We denote by $\d_{z_0}$
the delta-distribution at the point $z_0$. Consider the formal sum
$$\Theta=\sum_{\a,\b}A_{\a,\b}\pp^\a\overline\pp^\b\d_{z_0}$$
with some coefficients $A_{\a,\b}$. Each of the distributions
$\vartheta_{\a,\b}\simeq\pp^\a\overline\pp^\b\d_{z_0}$ defines a bounded (and compact) Toeplitz
operator $T_{\a,\b}$ with norm estimate $\|T_{\a,\b}\|\le C (1+|z_0|)^{\a+\b} \a!\b!$.
Therefore, if the coefficients $A_{\a,\b}$ decay so rapidly that the series $\sum
|A_{\a,\b}|(1+|z_0|)^{\a+\b} \a!\b!$ converges, the corresponding operator series $\sum
A_{\a,\b}T_{\a,\b}$ converges in the norm and thus defines a Toeplitz operator with symbol
supported at the point $z_0$ involving derivatives of all orders.
\end{example}

\end{document}